\documentclass[preprint,review,10pt]{amsart}
\usepackage{amsmath,amssymb,amsfonts,xspace}
\newtheorem{theorem}{Theorem}[section]

\theoremstyle{definition}
\newtheorem{definition}[theorem]{Definition}
\newtheorem{example}[theorem]{Example}

\newtheorem{corollary}[theorem]{Corollary}

\theoremstyle{remark}

\numberwithin{equation}{section}

\begin{document}

\title{ Classes of $F$-hyperideals in a Krasner $F^{(m,n)}$-hyperring
  }

\author{M. Anbarloei}
\address{Department of Mathematics, Faculty of Sciences,
Imam Khomeini International University, Qazvin, Iran.
}

\email{m.anbarloei@sci.ikiu.ac.ir }




\keywords{  prime $F$-hyperideal, maximal $F$-hyperideal, primary $F$-hyperideal, Krasner $F^{(m,n)}$-hyperring.}

\begin{abstract}
Krasner $F^{(m,n)}$-hyperrings were introduced and  investigated by Farshi and Davvaz. In this paper, our purpose is to  deﬁne and characterize  three particular classes of $F$-hyperideals in a Krasner $F^{(m,n)}$-hyperring, namely prime $F$-hyperideals, maximal $F$-hyperideals and  primary $F$-hyperideals, which extend similar concepts of ring context. Furthermore, we  examine the relations between these structures. Then a number of major conclusions are given to explain the general framework of
these structures.

\end{abstract}
\maketitle
\section{Introduction}
 As it is well known, the notion of a fuzzy set was introduced by Zadeh in 1965 \cite{zadeh}.  After the pioneering work of Zadeh,  the fuzzy sets have been used in the reconsideration of classical mathematics. A number of articles have applied fuzzy concepts to algebraic structures. Rosenfeld introduced and studied fuzzy sets in the context of group theory and formulated the notion of a fuzzy subgroup of a group \cite{Rosenfeld}. The notions of fuzzy subrings and ideals were defined by Liu \cite{liu}.  A considerable amount of work has been done on fuzzy ideals. In particular, many papers were written on prime fuzzy ideals.
 
Hyperstructure theory was born in 1934 when     Marty \cite{s1}, a French mathematician, defined  the concept of a hypergroup as a generalization of groups. Many papers and books concerning hyperstructure theory have
appeared in literature. For instance, you can see the papers \cite {s2, s3, davvaz1, davvaz2, s4}. The simplest algebraic hyperstructures which possess the properties of closure and associativity are called  semihypergroups. 
$n$-ary semigroups and $n$-ary groups are algebras with one $n$-ary operation which is associative and invertible in a generalized sense. The idea of investigations of $n$-ary algebras goes back to Kasner’s lecture \cite{s5} at the 53rd annual meeting of the American Association of the Advancement of Science in 1904.    In 1928, Dorente wrote the first paper concerning the theory of $n$-ary groups \cite{s6}. Later on, Crombez and
Timm \cite{s7, s8} defined  the notion of the $(m, n)$-rings and their quotient structures.  
 The $n$-ary hyperstructures  have been studied in  \cite{l1, l2, l3, ma, rev1}. In \cite{s9}, Davvaz and Vougiouklis introduced a generalization of the notion of a hypergroup in the sense of Marty and a generalization of an $n$-ary group,   which is called $n$-ary hypergroup. Mirvakili and Davvaz \cite{cons} defined $(m,n)$-hyperrings and obtained several results in this respect. 
One important class of hyperrings was introduced by Krasner, where the addition is a hyperoperation, while the multiplication is an ordinary binary operation, which is called Krasner hyperring.  In \cite{d1},  a generalization of the Krasner hyperrings, which is a subclass of $(m,n)$-hyperrings, was defined by Mirvakili and Davvaz. It is called Krasner $(m,n)$-hyperring. A Krasner $(m, n)$-hyperring is an algebraic hyperstructure $(R, f, g)$, or simply $R$,  which satisfies the following axioms:
(1) $(R, f$) is a canonical $m$-ary hypergroup;
(2) $(R, g)$ is a $n$-ary semigroup;
(3) the $n$-ary operation $g$ is distributive with respect to the $m$-ary hyperoperation $f$ , i.e., for every $a^{i-1}_1 , a^n_{ i+1}, x^m_ 1 \in R$, and $1 \leq i \leq n$,
$g(a^{i-1}_1, f(x^m _1 ), a^n _{i+1}) = f(g(a^{i-1}_1, x_1, a^n_{ i+1}),..., g(a^{i-1}_1, x_m, a^n_{ i+1}))$;
(4) $0$ is a zero element (absorbing element) of the $n$-ary operation $g$, i.e., for every $x^n_ 2 \in R$ we have 
$g(0, x^n _2) = g(x_2, 0, x^n _3) = ... = g(x^n_ 2, 0) = 0$. Ameri and Norouzi in \cite{sorc1} introduced some important
hyperideals such as Jacobson radical, n-ary prime and primary hyperideals, nilradical, and n-ary multiplicative subsets of Krasner $(m, n)$-hyperrings. For more study on Krasner $(m,n)$-hyperring refer to  \cite{mah2, mah3, asadi, rev2,  d1, nour,   rev1, Yassine}.

The connections between  hyperstructures and fuzzy sets is considered as many noticeable researches. The concept of $F$-polygroups was introduced and analyzed by Zahedi and Hasankhani in \cite{Zahedi1, Zahedi2}. The fuzzy hyperring notion was defined and studied in \cite {Leoreanu1}. In this regards Motameni and et. al. continued the study of  the notion of fuzzy hyperideals of a fuzzy hyperring. They defined and characterized prime fuzzy hyperideals and maximal fuzzy hyperideals and studied  the hyperideal transfer through a fuzzy hyperring homomorphism. Zhan and et. al. in \cite{zhan} concentrated on the quasi-coincidence of a fuzzy interval value with an interval valued fuzzy set. Davvaz in \cite{davvaz2} introduced the notion of a fuzzy hyperideal of a Krasner $(m,n)$-hyperring
and to extended the fuzzy results to Krasner $(m,n)$-hyperring.
 Let $G$ be an arbitrary set and  $L=[0,1]$ be the unit interval. Let $L^G$ (resp. $L^G_*)$ be the set of all  fuzzy subsets of $G$. An $F$-hyperoperation on $G$ is a function $\circ$ from $G \times G$ into $L^G_*$. If $\mu, \gamma \in L_*^G$ and $x \in G$, then $x \circ \mu=\bigcup_{a \in supp(\mu)}x \circ \mu$ and $\mu \circ \gamma=\bigcup_{a \in supp(\mu), b \in supp(\gamma)}a \circ b$ such that $supp(\mu)=\{a \in G \ \vert \mu(a) \neq 0\}$. The couple $(G,\circ)$ is called an $F$-polygroup if the following conditions
are satisfied: (1) $(a \circ b) \circ c=a \circ (b \circ c)$ for all $a,b,c \in G$, (2) there exists $e \in G$ with $a \in supp(a \circ e \cap e \circ a)$, for all $a \in G$, (3) for each $a \in G$, there exists a unique element $a^{-1} \in G$ with $e \in supp(a \circ a^{-1} \cap a^{-1} \circ a)$, (4) $c \in supp(a \circ b)$ implies that $a \in supp(c \circ b^{-1})$ implies that $b \in supp(a^{-1} \circ c)$, for all $a,b,c \in G$. Indeed, a fuzzy hyperoperation assigns to each pair of elements of $G$ a non-zero fuzzy subset of $G$, while a hyperoperation assigns to each pair of elements of $G$ a non-empty subset of $G$. \\
The concepts of Krasner $F^{(m,n)}$-hyperrings and $F$-hyperideals were defined in \cite{asli} by Farshi and Davvaz. 
In this paper, we continue the study of $F$-hyperideals of a Krasner $F^{(m,n)}$-hyperring, initiated in \cite{asli}. We  deﬁne and analyze there particular types of $F$-hyperideals in a Krasner $F^{(m,n)}$-hyperring, maximal $F$-hyperideals and  primary $F$-hyperideals. We investigate the connections between them. Moreover, we  introduce the concepts of $F$-radical, quotient Krasner $F^{(m,n)}$-hyperring and Jacobson radical. The overall
framework of these structures is then explained. It is shown  (Theorem \ref{akhar})  that if   $\mathfrak{Q}$ is an primary $F$-hyperideal of a Krasner $F^{(m,n)}$-hyperring $(\mathfrak{R},f,g)$, then $\sqrt{\mathfrak{Q}}^F$ is a prime $F$-hyperideal of $\mathfrak{R}$.
\section{Preliminaries}
In this section we recall some basic terms and definitions  from \cite{asli} which we need to develop our paper.\\
 A fuzzy subset of $G$ is a function  $\mu : G \longrightarrow L$ such that $L$ is the unit interval  $[0,1] \subseteq \mathbb{R}$. The set of all fuzzy subsets of $G$ is denoted by $L^G$ . Let $\mu , \gamma \in L^G$ and $\{\mu_\alpha \ \vert \ \alpha \in \Lambda\} \subseteq L^G$.  We define  the fuzzy subsets $\mu \cup \gamma$ and $\bigcup_{\alpha \in \Lambda} \mu_ \alpha$ as follows:
 \[(\mu \cup \gamma)(a)=max\{\mu(a),\gamma(a)\}\]
 and
 \[(\bigcup_{\alpha \in \Lambda} \mu_{\alpha})(a)=\bigvee_{\alpha \in \Lambda}\{\mu_\alpha(a)\}\]
For all $a \in G$. The set, $\{a \in G  \ \vert \ \mu(a) \neq 0
 \}$ is called the support of $\mu$ and is denoted by $supp(\mu)$. When $H \subseteq G$ and $t \in L$, we define $H_t \in L^G$ as follows:
\[
 H_t(a)=\left\{
 \begin{array}{lr}
 t&\text{if $a \in H$,}\\
 0&\text{if $a \notin H$.}
 \end{array} \right.\]
 In particular, if $H$ is a singleton, say $\{x\}$, then $\{x\}_t$ is  referred to as fuzzy point and, sometimes, denoted by $x_t$. The characteristic function of set $H$ is denoted by $\chi_H$. Let $L^G_*=L^G-\{0\}$. For a positive integer $n$, an $F^n$-hyperoperation on $G$ is a mapping $f$ from $G^n$ to $L^G_*$. This means that for any $a_1,\cdots, a_n \in G$, $f(a_1,\cdots,a_n)$ is a non-
zero fuzzy subset of $G$. If for all $a_1,\cdots,a_n \in G$, $supp(f(a_1,\cdots,a_n))$
is singleton, then $f$ is called an $F^n$-operation. 

Notice that the sequence $a_i, a_{i+1},..., a_j$ 
will be denoted by $a^j_i$. For $j< i$, $a^j_i$ is the empty symbol. Using this notation,

$\hspace{2cm}f(a_1,..., a_i, b_{i+1},..., b_j, c_{j+1},..., c_n)$ \\
will be written as $f(a^i_1, b^j_{i+1}, c^n_{j+1})$. The  expression will be written in the form $f(a^i_1, b^{(j-i)}, c^n_{j+1})$, when $b_{i+1} =... = b_j = b$ .

For $\mu_1^n \in L^G_*$, we define $f(\mu_1^n)$ as follows:
\[f(\mu_1^n)=\bigcup_{a_i \in supp(\mu_i)}f(a_1^n)\].
Let $a_1^n,a \in G$, $H \in P^*(G)$ and $\mu_1^n,\mu \in L^G_*$. Then, for $1 \leq i \leq n$

(1) $f(a_1^{i-1},\mu,a_{i+1}^n)$ denotes $f(\chi_{\{a_1\}},\dots, \chi_{\{a_{i-1}\}},\mu,\chi_{\{a_{i+1}\}},\dots, \chi_{\{a_n\}})$,

(2) $f(a_1^{i-1},H,a_{i+1}^n)$ denotes $f(\chi_{\{a_1\}},\dots, \chi_{\{a_{i-1}\}},\chi_H,\chi_{\{a_{i+1}\}},\dots, \chi_{\{a_n\}})$,

(3) $f(\mu_1^{i-1},a,\mu_{i+1}^n)$ denotes $f(\mu_1^{i-1},\chi_{\{a\}},\mu_{i+1}^n)$,

(4) $f(\mu_1^{i-1},H,\mu_{i+1}^n)$ denotes $f(\mu_1^{i-1},\chi_H,\mu_{i+1}^n)$.\\
If for every $1 \leq i < j \leq n$ and all $a_1^{2n-1} \in G$,

$\hspace{1cm}f(a^{i-1}_1, f(a_i^{n+i-1}), a^{2n-1}_{n+i}) = f(a^{j-1}_1, f(a_j^{n+j-1}), a_{n+j}^{2n-1}),$\\
then the $F^n$-hyperoperation ($F^n$-operation) $f$ is called associative. $G$ with the associative $F^n$-hyperoperation ($F^n$-operation)  is called $F^n$-semihypergroup  ($F^n$-semigroup).
\begin{definition}
Let $(G,f)$ be a $F^m$-semihypergroup. Suppose that $G$ is equipped with a unitary operation $^{-1}:G \longrightarrow G$. The couple $(G,f)$ is called a canonical $F^m$-hypergroup, if

(1) $G$ has an $F$-identity element, i.e., there exists an element $e \in G$ such that for every $a \in G$, $supp(f(a,e^{(n-1)}))=\{a\}$,

(2) $a \in supp(f(a_1^m))$ implies $a_i \in supp(f(a_1^{-1},\dots,a_{i-1}^{-1},a,a_{i+1}^{-1},\dots,a_n^{-1}))$, for all $a_1^m,a \in G$ and $1 \leq i \leq n$,

(3) for all $a_1^m$ and for all $\sigma \in \mathbb{S}_m$, $f(a_1^m)=f(a_{\sigma(1)}^{\sigma(m)})$.
\end{definition}
\begin{definition}
A Krasner $F^{(m,n)}$-hyperring is an algebraic hyperstructure $(\mathfrak{R},f,g)$, or simply $\mathfrak{R}$, which satisfies the following axioms:

(i) $(\mathfrak{R},f)$ is a canonical $F^m$-hypergroup,

(ii) $(\mathfrak{R},g)$ is an $F^n$-semigroup,

(iii) for every $x_1^{i-1},x_{i+1}^n, a_1^m \in \mathfrak{R}$ and $1 \leq i \leq n$, 
\[g(x_1^{i-1},f(a_1^m),x_{i+1}^n)=f(g(x_1^{i-1},a_1,x_{i+1}^n),\dots,g(x_1^{i-1},a_m,x_{i+1}^n)),\]

(iv) for every $a_2^n \in \mathfrak{R}$,
$supp(g(e,a_2^n))=\{e\}$ where $e$ is the $F$-identity element of $(\mathfrak{R},f)$.
\end{definition}
$\mathfrak{R}$ is called commutative Krasner $F^{(m,n)}$-hyperring if $g(a_1^n)=g(x_{\sigma(1)}^{\sigma(n)})$, for all $a_1^n \in \mathfrak{R}$ and for every $\sigma \in \mathbb{S}_n$. In the sequel, we assume that all Krasner $F^{(m,n)}$-hyperrings are commutative. We say that $\mathfrak{R}$ is with scalar $F$-identity if there exists an element $e^{\prime}$ such that $supp(g(a,e^{\prime^{(n-1)}}))=\{a\}$ for all $a \in \mathfrak{R}$.
\begin{definition}
Let $(\mathfrak{R},f,g)$ be a Krasner $F^{(m,n)}$-hyperring. A non-empty subset $S$ of $\mathfrak{R}$ is support closed under $f$ and $g$ whenever for all $a_1^m,b_1^n \in S$, $supp(f(a_1^m)) \subseteq S$ and $supp(g(b_1^n)) \subseteq S$. $S$ is called a Krasner $F$-subhyperring of $(\mathfrak{R},f,g)$ if $(S,f,g)$ is itself a Krasner $F^{(m,n)}$-hyperring. A Krasner $F$-subhyperring $\mathfrak{I}$ of $(\mathfrak{R},f,g)$ is said to be an  $F$-hyperideal if $supp(g(a_1^{i-1},\mathfrak{I},a_{i+1}^n)) \subseteq \mathfrak{I}$ for all $a_1^n \in \mathfrak{R}$ and $1 \leq i \leq n$.
\end{definition}
\begin{definition}
Let $(\mathfrak{R}_1, f_1, g_1)$ and $(\mathfrak{R}_2, f_2, g_2)$ be two Krasner $(m, n)$-hyperrings. A mapping
$h : \mathfrak{R}_1 \longrightarrow \mathfrak{R}_2$ is called a homomorphism if for all $x^m _1 \in \mathfrak{R}_1$ and $y^n_ 1 \in \mathfrak{R}_1$ we have

(1) $h(e_{\mathfrak{R}_1})=e_{\mathfrak{R}_2}$ such that $e_{\mathfrak{R}_1}$ and $e_{\mathfrak{R}_2}$ are F -identity elements of $\mathfrak{R}_1$ and $\mathfrak{R}_2$, respectively,

(2) $h(supp(f_1(x_1,..., x_m))) = supp(f_2(h(x_1),...,h(x_m)))$

(3) $h(supp(g_1(y_1,..., y_n))) =supp(g_2(h(y_1),...,h(y_n))). $
\end{definition}

\section{Prime F-hyperideals and maximal $F$-hyperideals}
We start this section by introducing the concept of prime $F$-hyperideals  of a Krasner $F^{(m,n)}$-hyperring $(\mathfrak{R},f,g)$.
\begin{definition}
An $F$-hyperideal $\mathfrak{P}$ of a Krasner $F^{(m,n)}$-hyperring $(\mathfrak{R},f,g)$ is called a prime $F$-hyperideal if for  all $\mu_1^n \in L^{\mathfrak{R}}_*$, $supp(g(\mu_1^n)) \subseteq \mathfrak{P}$ implies that $supp(\mu_i) \subseteq \mathfrak{P}$ for some $1 \leq i \leq n$.
\end{definition}

Our first theorem characterizes the prime $F$-hyperideals of a Krasner $F^{(m,n)}$-hyperring $(\mathfrak{R},f,g)$. 
\begin{theorem} \label{one}
Let $\mathfrak{P}$ is an $F$-hyperideal of a Krasner $F^{(m,n)}$-hyperring $(\mathfrak{R},f,g)$. Then $\mathfrak{P}$ is a prime $F$-hyperideal if and only if for all $a_1^n \in \mathfrak{R}$, $supp(g(a_1^n)) \subseteq \mathfrak{P}$ implies that $a_i \in \mathfrak{P}$ for some $1 \leq i \leq n$. 
\end{theorem}
\begin{proof}
$\Longrightarrow$ Let $\mathfrak{P}$ is a prime $F$-hyperideal of $\mathfrak{R}$. Let for $a_1^n \in \mathfrak{R}$, $supp(g(a_1^n)) \subseteq \mathfrak{P}$. Then $supp(g(\chi_{\{a_i\}},\dots,\chi_{\{a_n\}})) \subseteq \mathfrak{P}$. Since $\mathfrak{P}$ is a prime $F$-hyperideal of $\mathfrak{R}$, then we have $supp(\chi_{\{a_i\}})) \subseteq \mathfrak{P}$ for some $1 \leq i \leq n$. This implies that $\{a_i\} \subseteq \mathfrak{P}$ and so $a_i \in \mathfrak{P}$, as needed.\\
$\Longleftarrow$Let $supp(g(\mu_1^n)) \subseteq \mathfrak{P}$ for some $\mu_1^n \in L^{\mathfrak{R}}_*$. Let for all $1 \leq i \leq n$, $supp(\mu_i) \nsubseteq \mathfrak{P}$. Then for each $1 \leq i \leq n$, there exist an element $a_i \in supp(\mu_i)$ such that $a_i \notin \mathfrak{P}$. By Proposition 3.2 in \cite{asli}, $supp(g(a_1^n)) \subseteq supp(g(\mu_1^n))$. Therefore we have $supp(g(a_1^n)) \subseteq \mathfrak{P}$. By the hypothesis,  there exists $1 \leq i \leq n$ such that $a_i \in \mathfrak{P}$ and this is a contradiction. Thus $\mathfrak{P}$ is a prime $F$-hyperideal of $\mathfrak{R}$.
\end{proof}
\begin{example}
Let $(\mathbb{Z},+,\dot)$ be the ring of integers and $t_1,t_2 \in (0,1]$. We define an $F^{m}$-operation $f$ and an $F^n-operation$ $g$ on $\mathbb{Z}$ as follows:
\[f(a_1,\cdots,a_m)=(a_1+\cdots+a_m)_{t_1} \ \text{ for all} \ a_1,\cdots,a_m \in \mathbb{Z}\]
$\hspace{2cm}g(a_1,\cdots,a_n)=(a_1.\cdots.a_n)_{t_2} \hspace{0.8cm} \text{ for all} \   a_1,\cdots,a_n \in \mathbb{Z}.$\\
It is easy to verify that $(\mathbb{Z},f,g)$ is a Krasner $F^{(m,n)}$-hyperring. All $F$-hyperideals $p\mathbb{Z}$, where $p$ is a prime natural number are prime $F$-hyperideals of Krasner $F^{(m,n)}$-hyperring $(\mathbb{Z},f,g)$.
\end{example}
\begin{example}
Consider the Krasner $F^{(m,n)}$-hyperring $(G,f,g)$, given in Example 3.3 in \cite{asli}. Let $a \in G$. Then $\{e,a\}$ is a hyperideal of $G$ but it is not a prime $F$-hyperideal of $G$.
\end{example}

Let $\mathfrak{I}$ be a normal $F$-hyperideal of a Krasner $F^{(m,n)}$-hyperring $(\mathfrak{R},f,g)$.  The set $[\mathfrak{R}:\mathfrak{I}^*]=\{\mathfrak{I}^*[x] \ \vert \ x \in \mathfrak{R}\}$ is a Krasner $(m,n)$-hyperring with $m$-hyperoperation $f_{\vert I^*}$ and $n$-operation  $g_{\vert I^*}$ as follows:

\[f_{\vert I^*}({I^*}_{[x_1]}^{[x_n]})=\{I^*[x]\},\hspace{0.5cm} \forall x \in supp(f(x_1^m))\]

$\hspace{2.5cm}g_{\vert I^*}({I^*}_{[x_1]}^{[x_n]})=\{I^*[supp(g(x_1^n))]\}$.\\
Thus $[\mathfrak{R}:\mathfrak{I}^*]_i$ is a Krasner $F^{(m,n)}$-hyperring (for more details refer to \cite{asli}).
Now, we determine when the $F$-hyperideal $[\mathfrak{J}:\mathfrak{I}^*]_i$ of $[\mathfrak{R}:\mathfrak{I}^*]_i$ is  prime.
\begin{theorem}
Let $\mathfrak{I}$ be a normal $F$-hyperideal of a Krasner $F^{(m,n)}$-hyperring $(\mathfrak{R},f,g)$ and let  $\mathfrak{J}$ be an $F$-hyperideal of $(\mathfrak{R},f,g)$. If $\mathfrak{J}$ is a prime $F$-hyperideal of $\mathfrak{R}$, then $[\mathfrak{J}:\mathfrak{I}^*]_i$ is a prime $F$-hyperideal of $[\mathfrak{R}:\mathfrak{I}^*]_i$.
\end{theorem}
\begin{proof}
By Theorem 5.10 in \cite{asli}, $[\mathfrak{J}:\mathfrak{I}^*]_i$  is an $F$-hyperideal of $[\mathfrak{R}:\mathfrak{I}^*]_i$. Let for ${I^*}_{[x_1]}^{[x_n]} \in [\mathfrak{R}:\mathfrak{I}^*]_i$, $g_{\vert I^*}({I^*}_{[x_1]}^{[x_n]}) \subseteq [\mathfrak{J}:\mathfrak{I}^*]_i$. Then $I^*[supp(g(x_1^n))] \subseteq [\mathfrak{J}:\mathfrak{I}^*]_i$. This implies that $supp(g(x_1^n)) \subseteq \mathfrak{J}$. Since $\mathfrak{J}$ is a prime $F$-hyperideal of $\mathfrak{R}$, then there exist $1 \leq i \leq n$ such that $x_i \in \mathfrak{J}$. This means that $I^*[x_i] \in [\mathfrak{J}:\mathfrak{I}^*]_i$. Thus $[\mathfrak{J}:\mathfrak{I}^*]_i$ is a prime $F$-hyperideal of $[\mathfrak{R}:\mathfrak{I}^*]_i$.
\end{proof}
The following result investigates the stability of  prime $F$-hyperideals property under a transfer.
\begin{theorem}
Let $(\mathfrak{R}_1,f_1,g_1)$ and   $(\mathfrak{R}_2,f_2,g_2)$  be two Krasner $F^{(m,n)}$-hyperrings and $h : \mathfrak{R}_1 \longrightarrow \mathfrak{R}_1$ be a homomorphism. If $\mathfrak{P}$ is a prime $F$-hyperideal of $\mathfrak{R}_2$, then $h^{-1}(\mathfrak{P})$ is a prime $F$-hyperideal of $\mathfrak{R}_2$.
\end{theorem}
\begin{proof}
Suppose that $(\mathfrak{R}_1,f_1,g_1)$ and   $(\mathfrak{R}_2,f_2,g_2)$  are two Krasner $F^{(m,n)}$-hyperrings and $h : \mathfrak{R}_1 \longrightarrow \mathfrak{R}_1$ is a homomorphism. Let $supp(g_1(a_1^n)) \subseteq h^{-1}(\mathfrak{P)}$ for some $a_1^n \in \mathfrak{R}_1$. Then we obtain $h(supp(g_1(a_1^n))) \subseteq \mathfrak{P}$ which implies  $supp(g_2(h(a_1), \cdots,h(a_n)) \subseteq \mathfrak{P}$. Since $\mathfrak{P}$ is a prime $F$-hyperideal of $\mathfrak{R}_2$, there exists some $1 \leq i \leq n$ such that $h(a_i) \in \mathfrak{P}$. Therefore $a_i \in h^{-1}(\mathfrak{P})$, as needed.
\end{proof}
Let $(\mathfrak{R}_1,f_1,g_1)$ and   $(\mathfrak{R}_2,f_2,g_2)$  be two Krasner $F^{(m,n)}$-hyperrings and $\mathfrak{R}_1 \times \mathfrak{R}_2=\{(a,b) \ \vert \  a \in \mathfrak{R}_1, b\in \mathfrak{R}_2 \}$. Then, by Proposition 5.6 in \cite{asli}, $(\mathfrak{R}_1 \times \mathfrak{R}_2, f_\otimes,g_\otimes)$ is a Krasner $F^{(m,n)}$-hyperring, where
\[f_\otimes((a_1,b_1),\cdots,(a_m,b_m))(a,b)=min\{f_1(a_1^m)(a),f_2(b_1^m)(b)\}\]
\[g_\otimes((a_1,b_1),\cdots,(a_n,b_n))(a,b)=min\{g_1(a_1^n)(a),g_2(b_1^n)(b)\}\]
Now, we establish the following result.
\begin{theorem}
Let $(\mathfrak{R}_1,f_1,g_1)$ and   $(\mathfrak{R}_2,f_2,g_2)$  be two Krasner $F^{(m,n)}$-hyperrings. If $\mathfrak{P}_1$ is a prime $F$-hyperideal of $\mathfrak{R}_1$, then $\mathfrak{P}_1 \times \mathfrak{R}_2$ is a prime $F$-hyperideal of $\mathfrak{R}_1 \times \mathfrak{R}_2$.
\end{theorem}
\begin{proof}
We presume $(a_1,b_1),\cdots,(a_n,b_n) \in \mathfrak{R}_1 \times \mathfrak{R}_2$ with 

$\hspace{1 cm}supp(g_\otimes((a_1,b_1),\cdots,(a_n,b_n))) \subseteq \mathfrak{P}_1 \times \mathfrak{R}_2$.\\ This means that 

$\{(a,b) \ \vert \ g_\otimes((a_1,b_1),\cdots,(a_n,b_n))(a,b) \neq 0\}\subseteq \mathfrak{P}_1 \times \mathfrak{R}_2$ \\and so

 $\{(a,b) \ \vert \ min\{g_1(a_1^n)(a),g_2(b_1^n)(b)\} \neq 0\}\subseteq \mathfrak{P}_1 \times \mathfrak{R}_2$. \\This implies that $supp(g_1(a_1^n)) \subseteq \mathfrak{P}_1$. Since $\mathfrak{P}_1$ is a prime $F$-hyperideal of $\mathfrak{R}_1$, we get $a_i \in \mathfrak{P}_1$ for some $1 \leq i \leq n$. Thus we have $(a_i,b_i) \in \mathfrak{P}_1 \times \mathfrak{R}_2$. Consequently, $\mathfrak{P}_1 \times \mathfrak{R}_2$ is a prime $F$-hyperideal of $\mathfrak{R}_1 \times \mathfrak{R}_2$.
\end{proof}
Let $\mathfrak{I}$ be a $F$-hyperideal of a Krasner $F^{(m, n)}$-hyperring $(\mathfrak{R},f,g)$. Then the set

$\mathfrak{R}/\mathfrak{I}=\{supp(f(a_1^{i-1},\mathfrak{I},a_{i+1}^m))\ \vert \ a_1^{i-1},a_{i+1}^n \in \mathfrak{R}\}$\\
endowed with $F^m$-hyperoperation $f$ which for all $a_{11}^{1m},\dots,a_{m1}^{mm} \in \mathfrak{R}$

$f(supp(f(a_{11}^{1(i-1)},\mathfrak{I},a_{1(i+1)}^{1m})),\dots,supp(f(a_{m1}^{m(i-1)},\mathfrak{I},a_{m(i+1)}^{mm})))$

$\hspace{0.1cm}=supp(f(supp(f(a_{11}^{m1})),\dots,supp(f(a_{1(i-1)}^{m(i-1)})),\mathfrak{I},supp(f(a_{1(i+1)}^{m(i+1)})),\dots,supp(f(a_{1m}^{mm})))$\\
and with $F^n$-operation $g$ which for all $a_{11}^{1m},\dots,a_{n1}^{nm} \in \mathfrak{R}$

$g(supp(f(a_{11}^{1(i-1)},\mathfrak{I},a_{1(i+1)}^{1m})),\dots,supp(f(a_{n1}^{n(i-1)},\mathfrak{I},a_{n(i+1)}^{nm})))$

$\hspace{0.1cm}=supp(f(supp(g(a_{11}^{n1})),\dots,supp(g(a_{1(i-1)}^{n(i-1)})),\mathfrak{I},supp(g(a_{1(i+1)}^{n(i+1)})),\dots,supp(g(a_{1m}^{nm})))$\\
construct a Krasner $F^{(m,n)}$-hyperring, and $(\mathfrak{R}/\mathfrak{I},f,g)$ is called the quotient Krasner $F^{(m, n)}$-hyperring of $R$ by $\mathfrak{I}$.
\begin{theorem}
Let $\mathfrak{I}$ be an $F$-hyperideal of  a Krasner $F^{(m, n)}$-hyperring $(\mathfrak{R},f,g)$. Then the natural  map $\pi: \mathfrak{R} \longrightarrow \mathfrak{R}/\mathfrak{I}$, by $\pi(a)=supp(f(a,\mathfrak{I},e^{(m-2)}))$  is an epimorphism. 
\end{theorem}
\begin{proof}
Let $\mathfrak{I}$ be an $F$-hyperideal of  a Krasner $F^{(m, n)}$-hyperring $(\mathfrak{R},f,g)$. It is clear that $\pi$ is a projection map. We have to show that $\pi$ is a homomorphism. $\pi(e_R)=supp(f(e_R,\mathfrak{I},e_R^{(m-2)}))=\mathfrak{I}=e_{R/\mathfrak{I}}$, by Lemma 3.14 in \cite{asli}. For all $a_1^m \in \mathfrak{R}$,

$\pi(supp(f(a_1^m))=supp(f(supp(f(a_1^m)),\mathfrak{I},\{e\}^{(m-2)})$

$\hspace{2.4cm}=supp(f(supp(f(a_1^m)),\mathfrak{I},(supp(f(e^m))^{(m-2)}))$

$\hspace{2.4cm}=supp(f(supp(f(a_1,\mathfrak{I},e^{(m-2)})),\dots,supp(f(a_m,\mathfrak{I},e^{(m-2)}))$

$\hspace{2.4cm}=supp(f(\pi(a_1),\dots,\pi(a_m))$.\\
Furthermore, for all $b_1^n \in \mathfrak{R}$ we have

$\pi(supp(g(b_1^n))=supp(f(supp(g(b_1^n)),\mathfrak{I},\{e\}^{(m-2)})$

$\hspace{2.4cm}=supp(f(supp(g(b_1^n)),\mathfrak{I},(supp(g(e^n))^{(m-2)}))$

$\hspace{2.4cm}=supp(g(supp(f(b_1,\mathfrak{I},e^{(m-2)})),\dots,supp(f(b_n,\mathfrak{I},e^{(m-2)})))$

$\hspace{2.4cm}=supp(g(\pi(b_1),\dots,\pi(b_n))$.\\
Hence, $\pi$ is a homomorphism. 
\end{proof}
\begin{definition}
A Krasner $F^{(m,n)}$-hyperring $(\mathfrak{R},f,g)$ is called a hyperintegral $F$-domain, if for all $a_1^n \in \mathfrak{R}$, $supp(g(a_1^n))=\{e\}$ implies that $a_i=e$ for some $1 \leq i \leq n.$
\end{definition} 
The next theorem characterizes prime $F$-hyperideals in the sense of  quotient Krasner $F^{(m, n)}$-hyperrings .
\begin{theorem}
Let $\mathfrak{P}$ be an $F$-hyperideal of a Krasner $F^{(m,n)}$-hyperring $(\mathfrak{R},f,g)$. Then $\mathfrak{P}$ is prime if and only if $\mathfrak{R}/\mathfrak{P}$ is a hyperintegral $F$-domain.
\end{theorem}
\begin{proof}
$\Longrightarrow$ Let $\mathfrak{P}$ be a $F$-hyperideal of $\mathfrak{R}$ and for all $a_{11}^{1m},\dots,a_{n1}^{nm}$, 

$supp(f(a_{11}^{1(i-1)},\mathfrak{P}, a_{1(i+1)}^{1m}), \dots, supp(f(a_{n1}^{n(i-1)},\mathfrak{P}, a_{n(i+1)}^{nm}) \in \mathfrak{R}/\mathfrak{P}$\\
such that

$g(supp(f(a_{11}^{1(i-1)},\mathfrak{I},a_{1(i+1)}^{1m})),\dots,supp(f(a_{n1}^{n(i-1)},\mathfrak{I},a_{n(i+1)}^{nm})))=\mathfrak{P}=\{e_{\mathfrak{R}/\mathfrak{I}}\}.$
\\Then we get

$supp(f(supp(g(a_{11}^{n1})),\dots,supp(g(a_{1(i-1)}^{n(i-1)})),\mathfrak{I},supp(g(a_{1(i+1)}^{n(i+1)})),\dots,supp(g(a_{1m}^{nm})))=\mathfrak{P}$\\
and so

$supp(f(supp(g(a_{11}^{n1})),\dots,supp(g(a_{1(i-1)}^{n(i-1)})),e,supp(g(a_{1(i+1)}^{n(i+1)})),\dots,supp(g(a_{1m}^{nm}))) \subseteq \mathfrak{P}.$\\
This means that

$g(supp(f(a_{11}^{1(i-1)},e,a_{1(i+1)}^{1m})),\dots,supp(f(a_{n1}^{n(i-1)},e,a_{n(i+1)}^{nm}))) \subseteq \mathfrak{P}$\\ which imples 

$g(\chi_{supp(f(a_{11}^{1(i-1)},e,a_{1(i+1)}^{1m}))},\dots,\chi_{supp(f(a_{n1}^{n(i-1)},e,a_{n(i+1)}^{nm}))}) \subseteq \mathfrak{P}$. \\Since $\mathfrak{P}$ is a prime $F$-hyperideal of $\mathfrak{R}$, we obtain $supp(\chi_{supp(f(a_{j1}^{j(i-1)},e,a_{j(i+1)}^{jm}))}) \subseteq \mathfrak{P}$ for some $1 \leq j \leq n$. Therefore $supp(f(a_{j1}^{j(i-1)},e,a_{j(i+1)}^{jm})) \subseteq \mathfrak{P}$. Then we get $supp(f(a_{j1}^{j(i-1)},\mathfrak{P},a_{j(i+1)}^{jm})) = \mathfrak{P}$, by Lemma 3.14 (1) in \cite{asli}. Consequently, $\mathfrak{R}/\mathfrak{P}$ is a hyperintegral $F$-domain.\\
$\Longleftarrow$ Let $\mathfrak{R}/\mathfrak{P}$ be a hyperintegral $F$-domain. Suppose that $supp(g(a_1^n)) \subseteq \mathfrak{P}$ for all $a_1^n \in \mathfrak{R}$. Then we have $supp(f(supp(g(a_1^n)),\mathfrak{P},e^{(m-2)}))=\mathfrak{P}$ by Lemma 3.14 (1) in \cite{asli}. Therefore 

$supp(f(supp(g(a_1^n)),\mathfrak{P},(supp(g(e^n))^{(m-2)}))=\mathfrak{P}$. \\Then by the definition of the quotient Krasner $F^{(m, n)}$-hyperring we have

 $g(supp(f(a_1,\mathfrak{P},e^{(m-2)}),\dots,supp(f(a_n,\mathfrak{P},e^{(m-2)})))=\mathfrak{P}=\{e_{\mathfrak{R}/\mathfrak{I}}\}$.\\ Since $\mathfrak{R}/\mathfrak{P}$ is a hyperintegral $F$-domain, then we get $supp(f(a_i,\mathfrak{P},e^{(m-2)}))=\mathfrak{P}$ for some $1 \leq i \leq n$ which implies $a_i \in \mathfrak{P}$. consequently,  $\mathfrak{P}$ is a prime $F$-hyperideal of $\mathfrak{R}$.
\end{proof}
\begin{theorem}
Let $(\mathfrak{R},f,g)$ be a Krasner $F^{(m,n)}$-hyperring with the scalar identity $e^\prime$. Then for all $x \in \mathfrak{R}$, the set $\{a \in supp(g(r,x,e^{\prime^{(n-2)}}))\ \vert \ r \in \mathfrak{R}\}$ is an $F$-hyperideal of $\mathfrak{R}$. We say that the $F$-hyperideal is the hyperideal generated by $x$ and it is denoted by $<x>_F$.
\end{theorem}
\begin{proof}
We first show $<x>_F$ is support closed under $f$ and  $g$. Let $a_1^m \in <x>_F$. Then for each $1 \leq i \leq m$ there exist $r_i \in \mathfrak{R}$ such that $a_i \in supp(g(r_i,x,e^{\prime^{(n-2)}}))$. By Proposition 3.2 in \cite{asli},  we get

$supp(f(a_1^m))=supp(f(g(r_1,x,e^{\prime^{(n-2)}}),\dots,g(r_m,x,e^{\prime^{(n-2)}})))$

$\hspace{2cm}=supp(g(f(r_1,\dots,r_m),x,e^{\prime^{(n-2)}})$

$\hspace{2cm}=\bigcup _{r \in supp(f(r_1,\dots,r_m))}supp(g(r,x,e^{\prime^{(n-2)}}))$.\\
Since $r \in supp(f(r_1,\dots,r_m)) \subseteq \mathfrak{R}$, then $supp(f(a_1^m)) \subseteq <x>_F$. Let $b_1^n \in <x>_F$. Then for each $1 \leq i \leq n$ there exist $r_i \in \mathfrak{R}$ such that $b_i \in supp(g(r_i,x,e^{\prime^{(n-2)}}))$. Therefore  we have

$supp(g(b_1^n))=supp(g(g(r_1,x,e^{\prime^{(n-2)}}),\dots,g(r_n,x,e^{\prime^{(n-2)}})))$

$\hspace{2cm}=supp(g(g(r_1,\dots,r_n),x,e^{\prime^{(n-2)}}))$

$\hspace{2cm}=supp(g(r,x,e^{\prime^{(n-2)}}))$\\
such that $r \in supp(g(r_1,\dots,r_n)) \subseteq \mathfrak{R}$. Hence $supp(g(b_1^n)) \subseteq <x>_F$. It is easy to see that $(<x>_F,f,g)$ is a $F$-subhyperring of $\mathfrak{R}$. Now we show that  for all $r_1^n \in  \mathfrak{R}$ and $1 \leq i \leq n$, $supp(g(r_1^{i-1},<x>,r_{i+1}^n)) \subseteq <x>_F$. Let $r_1^n \in \mathfrak{R}$. Then

$supp(g(r_1^{i-1},<x>_F,r_{i+1}^n))=supp(g(r_1^{i-1},\chi_{<x>_F},r_{i+1}^n))$

$\hspace{4.1cm}=\bigcup_{a \in supp(\chi_{<x>_F})}supp(g(r_1^{i-1},a,r_{i+1}^n))$

$\hspace{4.1cm}=\bigcup_{a \in <x>_F}supp(g(r_1^{i-1},a,r_{i+1}^n))$

$\hspace{4.1cm} \subseteq \bigcup_{r \in \mathfrak{R}}supp(g(r_1^{i-1},g(r,x,e^{\prime^{(n-2)}}),r_{i+1}^n))$

$\hspace{4.1cm} = \bigcup_{r \in \mathfrak{R}}supp(g(g(r_1^{i-1},r,r_{i+1}^n),x,e^{\prime^{(n-2)}}))$

$\hspace{4.1cm} = \bigcup_{r \in \mathfrak{R}}\bigcup_{r^\prime \in supp(g(r_1^{i-1},r,r_{i+1}^n))}supp(g(r^\prime,x,e^{\prime^{(n-2)}}))$

$\hspace{4.1cm} \subseteq <x>_F$.\\
Thus $F$-subhyperring  $<x>_F=\{a \in supp(g(r,x,e^{\prime^{(n-2)}}))\ \vert \ r \in \mathfrak{R}\}$ of $\mathfrak{R}$ is an $F$-hyperideal of $\mathfrak{R}$.
\end{proof}
\begin{definition}
An $F$-hyperideal $\mathfrak{M}$ of a Krasner $F^{(m,n)}$-hyperring is called maximal if for every $F$-hyperideal $\mathfrak{N}$ of $\mathfrak{R}$, $\mathfrak{M} \subseteq \mathfrak{N} \subseteq \mathfrak{R}$ implies that $\mathfrak{N}=\mathfrak{M}$ or $\mathfrak{N}=\mathfrak{R}$.
\end{definition}
The intersection of all maximal $F$-hyperideals of $\mathfrak{R}$ is called Jacobson radical of $\mathfrak{R}$ and it is denoted by $J(\mathfrak{R})$. If $\mathfrak{R}$ does not have any maximal $F$-hyperideal, we let $J(\mathfrak{R})=\mathfrak{R}$.
\begin{example}
Let us consider the Krasner $F^{(m,2)}$-hypering $(G,f,g)$, given in Example 3.4 in \cite{asli}. Then $\{e\}$ is a maximal $F$-hyperideal of $G$.
\end{example}
The following theorem ensure that there is always a sufficient supply of the maximal $F$-hyperideals.
\begin{theorem}
Every  Krasner $F^{(m,2)}$-hypering $\mathfrak{R}$
with scalar $F$-identity $e^\prime$, has at least one maximal $F$-hyperideal.
\end{theorem}
\begin{proof}
The proof is the same as the proof in the classical context of maximal ideals of rings.
\end{proof}
We say that an element $x \in \mathfrak{R}$ is $F$-invertible if there exists $y \in \mathfrak{R}$ such that $supp(g(x,y,e^{\prime^{(n-2)}}))=\{e^\prime\}$. Also,
the subset $U$ of $\mathfrak{R}$ is $F$-invertible if and only if every element of $U$ is $F$-invertible.
\begin{theorem}
Let $\mathfrak{I}$ be an $F$-hyperideal of a Krasner $F^{(m,n)}$-hyperring with the scalar $F$-identity $e^\prime$. If every $a \in \mathfrak{R}-\mathfrak{M}$ is $F$-invertible in $\mathfrak{R}$, then $\mathfrak{M}$ is the only maximal
hyperideal of $\mathfrak{R}$.
\end{theorem}
\begin{proof}
Straight forward
\end{proof}
The Jacobson radical of a Krasner $F^{(m,n)}$-hyperring $\mathfrak{R}$ can be characterized as follows:
\begin{theorem} \label{har}
Let $\mathfrak{I}$ be an $F$-hyperideal of a Krasner $F^{(m,n)}$-hyperring with the scalar $F$-identity $e^\prime$. Then every element of $supp(f(e^\prime,\mathfrak{I},e^{(m-2)}))$ is $F$-invertible if and only if $\mathfrak{I} \subseteq J(\mathfrak{R})$.
\end{theorem}
\begin{proof}
$\Longrightarrow$ Suppose that every element of $supp(f(e^\prime,\mathfrak{I},e^{(n-2)}))$ is $F$-invertible. Let $\mathfrak{I} \nsubseteq J(\mathfrak{R})$. Then there exists a maximal  $F$-hyperideal $\mathfrak{M}$ of $\mathfrak{R}$ such that $\mathfrak{I} \nsubseteq \mathfrak{M}$. Let $x \in \mathfrak{I}$ but $x \notin \mathfrak{M}$. By Lemma 3.12 in \cite{asli}, $supp(f(\mathfrak{M},<x>_F,e^{(m-2)}))$ is an $F$-hyperideal of $\mathfrak{R}$. Since $\mathfrak{M} \subseteq  supp(f(\mathfrak{M},<x>_F,e^{(m-2)}))$ and $\mathfrak{M}$ is a maximal $F$-hyperideal of $\mathfrak{R}$, we have $  supp(f(\mathfrak{M},<x>_F,e^{(m-2)}))=\mathfrak{R}$ and so $e^\prime \in supp(f(\mathfrak{M},<x>_F,e^{(m-2)}))$. This means that there exist $m  \in supp(\chi_{\mathfrak{M}})=\mathfrak{M}$ and $a \in supp(\chi_{<x>_F})=<x>_F$ such that $e^\prime \in supp(f(m,a,e^{(m-2)}))$.  Since $(\mathfrak{R},f)$ is a canonical $F^m$-hypergroup, then $ m \in supp(f(e^\prime,-a,e^{(m-2)}))$. Since $supp(f(e^\prime,-a,e^{(m-2)})) \subseteq supp(f(e^\prime,-g(r,x,e^\prime),e^{(m-2)}))$ for some $r \in \mathfrak{R}$, we have $ m \in supp(f(e^\prime,g(r,x,e^\prime),e^{(m-2)}))$ which implies $ m \in supp(f(e^\prime,\mathfrak{I},e^{(m-2)}))$. This $m$ is $F$-invertible, a contradiction.\\
$\Longleftarrow$ Suppose that $\mathfrak{I} \subseteq J(\mathfrak{R})$. Let $a \in supp(f(e^\prime,\mathfrak{I},e^{(m-2)}))$ is not $F$-invertible. Then there exists $ x \in \mathfrak{I}$ such that  $a \in supp(f(e^\prime,x,e^{(m-2)}))$. We have $a \in \mathfrak{M}$ for some maximal $F$-hyperideal $\mathfrak{M}$, because $a$ is not $F$-invertible. From $a \in supp(f(e^\prime,x,e^{(m-2)}))$, it follows that $e^\prime \in supp(f(a,-x,e^{(m-2)})) \subseteq \mathfrak{M}$, a contradiction. Thus every element of $supp(f(e^\prime,\mathfrak{I},e^{(m-2)}))$ is $F$-invertible.
\end{proof}
In view of Theorem \ref{har}, we have the following result.
\begin{corollary}
Let $\mathfrak{M}$ be a maximal  $F$-hyperideal of a Krasner $F^{(m,n)}$-hyperring with the scalar $F$-identity $e^\prime$. If  every element of $supp(f(e^\prime,\mathfrak{M},e^{(m-2)}))$ is $F$-invertible, then $\mathfrak{M}$ is the only maximal
hyperideal of $\mathfrak{R}$.
\end{corollary}
\begin{theorem} \label{max}
Suppose that  $T$ is a non-empty subset of a  Krasner $F^{(m, n)}$-hyperring $(\mathfrak{R}, f , g)$ that is support closed under $g$ and $\mathfrak{I}$ is an $F$-hyperideal of $\mathfrak{R}$ such that $\mathfrak{I} \cap T =\varnothing$ . Then there exists an $F$-hyperideal $\mathfrak{P}$ which is maximal in the set of all hyperideals of $\mathfrak{R}$ disjoint from $T$ containing $\mathfrak{I}$. Furthermore any such $F$-hyperideal is  prime.
\end{theorem}
\begin{proof}
Let $\Sigma$ be the set of all hyperideals
of $\mathfrak{R}$ disjoint from $T$ containing $\mathfrak{I}$. Since $\mathfrak{I} \in \Sigma$, then $\Sigma \neq \varnothing$. Thus $\Sigma$ is a partially ordered set with respect to set inclusion relation. Then  there is an $F$-hyperideal $\mathfrak{P}$  which is maximal in $\Sigma$, by Zorn$^,$s lemma. Our task now is to show that $\mathfrak{P}$ is a prime $F$-hyperideal of $\mathfrak{R}$. Let $supp(g(a_1^n)) \subseteq \mathfrak{P}$ for some $a_1^n \in \mathfrak{R}$ such that for all $1 \leq i \leq n$, $a_i \notin \mathfrak{P}$. Then for each $1 \leq i \leq n$, $\mathfrak{P} \subseteq supp(f(\mathfrak{P},<a_i>,e^{(m-2)}))$. By maximality of $\mathfrak{P}$, we conclude that $supp(f(\mathfrak{P},<a_i>,e^{(m-2)})) \cap T \neq \varnothing$. Hence there exist $p_1^n \in \mathfrak{P}$ and $x_i \in <a_i>$ such that $supp(f(p_i,x_i,e^{(m-2)})) \cap T \neq \varnothing$ for each $1 \leq i \leq n$. Since $T$ is  support closed under $g$, then  there exist $r_i \in supp(f(p_i,x_i,e^{(m-2)}))$ for each $1 \leq i \leq n$ such that $supp(g(r_1^n)) \cap T \neq \varnothing$. Therefore \\

$supp(g(r_1^n)) \subseteq supp(g(f(p_1,x_1,e^{(m-2)})),\dots,f(p_n,x_n,e^{(m-2)})))$

$\hspace{1.9cm}=supp(f(g(p_1^n),g(p_1^{n-1},x_n),\cdots,g(p_1,x_2^n), \cdots, g(x_1^n),e^{(m-2^n)}))$

$\hspace{1.9cm}\subseteq \mathfrak{P}.$\\
This means that $\mathfrak{P} \cap T \neq \varnothing$ which is contradiction with $\mathfrak{P} \in \Sigma$. Thus, by Theorem \ref{one}, $\mathfrak{P}$ is a prime $F$-hyperideal of $\mathfrak{R}$.
\end{proof}
\begin{definition}
Let $\mathfrak{I}$ be an $F$-hyperideal of a  Krasner $F^{(m, n)}$-hyperring $(\mathfrak{R}, f , g)$ with
scalar $F$-identity $e^\prime$. 
The intersection of all  prime $F$-hyperideals  of $\mathfrak{R}$  containing $\mathfrak{I}$ is called $F$-radical of $\mathfrak{I}$, being denoted by $\sqrt{\mathfrak{I}}^{F}$. If $\mathfrak{R}$ does not have any prime $F$-hyperideal containing $\mathfrak{I}$, we define  $\sqrt{\mathfrak{I}}^{F}=\mathfrak{R}$.
\end{definition}
The following theorem gives an alternative definition of $\sqrt{\mathfrak{I}}^{F}$.
\begin{theorem}
$\mathfrak{I}$ be an $F$-hyperideal of a  Krasner $F^{(m, n)}$-hyperring $(\mathfrak{R}, f , g)$ with
scalar $F$-identity $e^\prime$. Then

$
\sqrt{\mathfrak{I}}^F= \biggm{\{} a \in \mathfrak{R} \ \vert \ \biggm{\{}
 \begin{array}{lr}
 supp(g(a^{(s)},e^{\prime^{(n-s)}}))\subseteq \mathfrak{I},& s \leq n\\
 supp(g_{(l)}(a^{(s)}))\subseteq \mathfrak{I} & s>n, s=l(n-1)+1
 \end{array}
 \biggm{\}} \biggm{\}}.$
 \end{theorem}
 \begin{proof}
 Let $a \in \sqrt{\mathfrak{I}}^F$ and let $\mathfrak{P}$ be a prime $F$-hyperideal of $\mathfrak{R}$ with $\mathfrak{I} \subseteq \mathfrak{P}$. Thus there exists $s \in \mathbb{N}$ with $ supp(g(a^{(s)},e^{\prime^{(n-s)}}))\subseteq \mathfrak{I}$ for $s \leq n$, or 
 $supp(g_{(l)}(a^{(s)}))\subseteq \mathfrak{I}$ for $ s=l(n-1)+1$. In the first case, we have $ supp(g(a^{(s)},e^{\prime^{(n-s)}}))\subseteq \mathfrak{P}$ and so $ supp(g(a,g(a^{(s-1)},e^{\prime^{(n-s+1)}}),e^{\prime^{(n-2)}}))\subseteq \mathfrak{P}$. Since $\mathfrak{P}$ is a prime $F$-hyperideal of $\mathfrak{R}$,  we get $a \in \mathfrak{P}$ or $supp(g(a^{(s-1)},e^{\prime^{(n-s+1)}})) \subseteq \mathfrak{P}$.From $supp(g(a^{(s-1)},e^{\prime^{(n-s+1)}})) \subseteq \mathfrak{P}$, it follows that $supp(g(a,g(a^{(s-2)},e^{\prime^{(n-s+2)}}),e^{\prime^{(n-2)}}) \subseteq \mathfrak{P}$ which implies $a \in \mathfrak{P}$ or $supp(g(a^{(s-2)},e^{\prime^{(n-s+2)}})) \subseteq \mathfrak{P}$. By continuing this process, we obtain $a \in  \mathfrak{P}$. Hence we have $a \in  \mathfrak{P}$ for all $\mathfrak{I} \subseteq \mathfrak{P}$ and so $a \in \bigcap_{\mathfrak{I} \subseteq \mathfrak{P}}\mathfrak{P}$. This means $\sqrt{\mathfrak{I}}^F \subseteq \bigcap_{\mathfrak{I} \subseteq \mathfrak{P}}\mathfrak{P}$. In the second case, we get $supp(g(g(\dots g(g(a^{(n)}),a^{(n-1)}),\dots),a^{(n-1)})) \subseteq \mathfrak{P}$. By using a similar argument , we get $\sqrt{\mathfrak{I}}^F \subseteq \bigcap_{\mathfrak{I} \subseteq \mathfrak{P}}\mathfrak{P}$. Now, suppose that $a \in \bigcap_{\mathfrak{I} \subseteq \mathfrak{P}}\mathfrak{P}$ but $a \notin \sqrt{\mathfrak{I}}^F$. Hence we conclude that for every $s \in \mathbb{N}$, $supp(g(a^{(s)},e^{\prime^{(n-s)}})) \nsubseteq \mathfrak{I}$. Let $T=\{e^\prime,a\} \cup \{r \in supp(g(a^{(t)},e^{\prime^{(n-t)}})) \ \vert \ 2 \leq t \}$. Clearly,   $T$ is a subset of $\mathfrak{R}$ that is support closed under $g$ and $T \cap \mathfrak{I}=\varnothing$. By Theorem \ref{max}, there exists a prime $F$-hyperideal $\mathfrak{P}$ with $\mathfrak{I} \subseteq \mathfrak{P}$ and $T \cap \mathfrak{P}=\varnothing$. This means $a \notin \mathfrak{P}$. This is contradiction as $a \in \bigcap_{\mathfrak{I} \subseteq \mathfrak{P}}\mathfrak{P}$. Thus $a \in \sqrt{\mathfrak{I}}^F$. Consequently, $\sqrt{\mathfrak{I}}^F=\bigcap_{\mathfrak{I} \subseteq \mathfrak{P}}\mathfrak{P}$.
 \end{proof}
\section{Primary $F$-hyperideals}
In this section, we aim to  present the definition of  primary $F$-hyperideals in a Krasner $F^{(m,n)}$-hyperring $(\mathfrak{R},f,g)$ and give some basic  properties of them.  
\begin{definition}
An $F$-hyperideal $\mathfrak{Q}$ of a Krasner $F^{(m,n)}$-hyperring $(\mathfrak{R},f,g)$ (with scalar $F$-identity $e^\prime$) is called a primary $F$-hyperideal if for  all $\mu_1^n \in L^{\mathfrak{R}}_*$, $supp(g(\mu_1^n)) \subseteq \mathfrak{Q}$ implies that $supp(\mu_i) \subseteq \mathfrak{Q}$ or $supp(g(\mu_1^{i-1},\chi_{\{e^\prime\}},\mu_{i+1}^n)) \subseteq \sqrt{\mathfrak{Q}}^F$ for some $1 \leq i \leq n$.
\end{definition}
\begin{theorem}
Let  $\mathfrak{Q}$ be an $F$-hyperideal of a Krasner $F^{(m,2)}$-hyperring $(\mathfrak{R},f,g)$ (with scalar $F$-identity $e^\prime$). Then $\mathfrak{Q}$ is primary if and only if for all $a_1^2 \in\mathfrak{R}$, $supp(g(a_1^2)) \subseteq \mathfrak{Q}$ implies that $a_1 \in \mathfrak{Q}$ or $a_2 \in \sqrt{\mathfrak{Q}}^F$.
\end{theorem}
\begin{proof}
$\Longrightarrow$ Let $\mathfrak{Q}$ is a primary $F$-hyperideal of $\mathfrak{R}$. Suppose that $supp(g(a_1^2)) \subseteq \mathfrak{Q}$ for some $a_1^2 \in\mathfrak{R}$. Since $supp(g(a_1^2)) =supp(g(\chi_{\{a_1\}},\chi_{\{a_2\}} ))$, then we have $supp(g(\chi_{\{a_1\}},\chi_{\{a_2\}} )) \subseteq \mathfrak{Q}$. Since $\mathfrak{Q}$ is a primary $F$-hyperideal of $\mathfrak{R}$, we get $supp(\chi_{\{a_1\}}) \subseteq \mathfrak{Q}$ or $supp(\chi_{\{a_2\}}) \subseteq \sqrt{\mathfrak{Q}}^F$. This means $a_1 \in \mathfrak{Q}$ or $a_2 \in \sqrt{\mathfrak{Q}}^F$.\\
$\Longrightarrow$ Let $supp(g(\mu_1^2)) \subseteq \mathfrak{Q}$ for some $\mu_1^2 \in L^{\mathfrak{R}}_*$ such that neither $supp(\mu_1) \subseteq \mathfrak{Q}$ nor $supp(\mu_2) \subseteq \sqrt{\mathfrak{Q}}^F$. Suppose that $x_1 \in supp(\mu_1) - \mathfrak{Q}$ and $x_2 \in supp(\mu_2) - \sqrt{\mathfrak{Q}}^F$. Clearly, $supp(g(x_1^2)) \subseteq  supp(g(\mu_1^2)) \subseteq \mathfrak{Q}$. By the hypothesis, we get $x_1 \in \mathfrak{Q}$ or $x_2 \in \sqrt{\mathfrak{Q}}^F$, a contradiction. 
\end{proof}
\begin{corollary}
Let  $\mathfrak{Q}$ be an $F$-hyperideal of a Krasner $F^{(m,n)}$-hyperring $(\mathfrak{R},f,g)$ (with scalar $F$-identity $e^\prime$). Then $\mathfrak{Q}$ is primary if and only if for all $a_1^n \in\mathfrak{R}$, $supp(g(a_1^n)) \subseteq \mathfrak{Q}$ implies that $a_i \in \mathfrak{Q}$ or $supp(g(a_1^{i-1},e^\prime,a_{i+1}^n)) \subseteq \sqrt{\mathfrak{Q}}^F$ for some $1 \leq i \leq n$.
\end{corollary}
The following is a direct consequence and can be proved easily and
so the proof is omited.
\begin{theorem} \label{prime}
If $\mathfrak{P}$ is a prime $F$-hyperideal of $\mathfrak{R}$, then $\mathfrak{P}$ is a primary $F$-hyperideal of $\mathfrak{R}$.
\end{theorem}
The next example shows that the inverse of Theorem \ref{prime} is not true, in general.
\begin{example} Suppose that $R=[0,1]$ and $t \in (0,1]$. Then $(R,f,g)$ is a Krasner $F^{2,3}$-hyperring, where $f$ and $g$ defined by
\[
f(a,b)= \bigg{\{}
 \begin{array}{lr}
 \chi_{max\{a,b\}}& \text{if} \ a \neq b \\
 \chi_{[0,a]} &  \text{if} \ a=b
 \end{array}\]
 and
 
 $\hspace{3.2cm}g(a,b,c)=(a.b.c)_t$\\
for all $a,b,c \in R$. The $F$-hyperideal $I=[0,0.5]$ is a priamary $F$-hyperideal of $R$. But it is not a prime $F$-hyperideal of $R$ as the fact that $supp(g(0.8,0.7,.0.6)) \subseteq I$ but $0.8,0.7,0.6 \notin I$.
\end{example}
In next theorem, we establish a relationship between prime $F$-hyperideals and primary $F$-hyperideals of a Krasner $F^{(m,n)}$-hyperring $(\mathfrak{R},f,g)$.
\begin{theorem} \label{akhar}
Let  $\mathfrak{Q}$ be an $F$-hyperideal of a Krasner $F^{(m,n)}$-hyperring $(\mathfrak{R},f,g)$ (with scalar $F$-identity $e^\prime$). If $\mathfrak{Q}$ is primary, then $\sqrt{\mathfrak{Q}}^F$ is a prime $F$-hyperideal of $\mathfrak{R}$.
\end{theorem}
\begin{proof}
Let $supp(g(x_1^n)) \subseteq \sqrt{\mathfrak{Q}}^F$ for some $x_1^n \in \mathfrak{R}$ such that $x_1^{i-1},x_{i+1}^n \notin \sqrt{\mathfrak{Q}}^F$. Our task now is to show that $a_i \in \sqrt{\mathfrak{Q}}^F$. Let $x \in supp(g(x_1^n)$. From $supp(g(x_1^n)) \subseteq \sqrt{\mathfrak{Q}}^F$, it follows that there exists $s \in \mathbb{N}$ such that if $s \leq n$, then $supp(g(x^{(s)},e^{\prime^{(n-s)}}  )) \subseteq \mathfrak{Q}$.   Therefore we have $supp(g(g(x_1^n)^{(s)},e^{\prime^{(n-s)}}  )) \subseteq \mathfrak{Q}$, because $supp(g(x^{(s)},e^{\prime^{(n-s)}}  ))= supp(g(g(x_1^n)^{(s)},e^{\prime^{(n-s)}}  ))$. Thus $supp(g(x_i^{(s)},g(x_1^{i-1},e^\prime,x_{i+1}^n)^{(s)},e^{\prime^{(n-2s)}})) \subseteq  \mathfrak{Q}$ and so $supp(g(g(x_i^{(s)},e^{\prime^{(n-s)}}),g(g(x_1^{i-1},e^\prime,x_{i+1}^n)^{(s)},e^{\prime^{(n-s)}}),e^{\prime^{(n-2)}})) \subseteq  \mathfrak{Q}$. Since $\mathfrak{Q}$ is a primary $F$-hyperideal of $\mathfrak{R}$, we get $supp(g(g(x_1^{i-1},e^\prime,x_{i+1}^n)^{(s)},e^{\prime^{(n-s)}})) \subseteq \mathfrak{Q}$ or $supp(g(x_i^{(s)},e^{\prime^{(n-s)}})) \subseteq \sqrt{\mathfrak{Q}}^F$. Suppose that  $supp(g(g(x_1^{i-1},e^\prime,x_{i+1}^n)^{(s)},e^{\prime^{(n-s)}})) \subseteq \mathfrak{Q}$. Then we get $supp(g(g(x_1^{(s)},e^{\prime^{(n-s)}}), g(x_2^{i-1},e^{\prime^2},x_{i+1}^n)^{(s)},e^{\prime^{(n-s-1)}})) \subseteq \mathfrak{Q}$ which means $supp(g(x_1^{(s)},e^{\prime^{(n-s)}})) \subseteq \mathfrak{Q}$ or $supp(g(g(x_2^{i-1},e^{\prime^2},x_{i+1}^n)^{(s)},e^{\prime^{(n-s)}})) \subseteq \sqrt{\mathfrak{Q}}^F$. Since $x_1 \notin \sqrt{\mathfrak{Q}}^F$, then $supp(g(g(x_2^{i-1},e^{\prime^2},x_{i+1}^n)^{(s)},e^{\prime^{(n-s)}})) \subseteq \sqrt{\mathfrak{Q}}^F$. By continuing this process, since $x_1^{i-1},x_{i+1}^n \notin \sqrt{\mathfrak{Q}}^F$, we have $supp(g(x_i^{(s)},e^{\prime^{(n-s)}})) \subseteq \sqrt{\mathfrak{Q}}^F$. By using the definition of $F$-radical of $\mathfrak{Q}$, we conclude that $a_i \in \sqrt{\mathfrak{Q}}^F$. If $s=l(n-1)+1$, then by using a similar argument , one can easily complete the proof.
\end{proof}
\section{Conclusion}
In this paper, our purpose is to extend the study initiated in \cite{asli} about Krasner $F^{(m,n)}$-hyperrings by Farshi and Davvaz. We defined prime $F$-hyperideals, maximal $F$-hyperideals and primary $F$-hyperideals of a Krasner $F^{(m,n)}$-hyperring $\mathfrak{R}$. We obtain many specific results explaining the structures. The future work can be on defining the concept of   $\delta$-primary $F$-hyperideals  unifing the notions of the prime and primary $F$-hyperideal in a frame. 

\end{document}